\documentclass{amsart}
\usepackage[utf8]{inputenc}
\usepackage{amssymb,tikz}

\newtheorem{theorem}{Theorem}
\newtheorem{lemma}{Lemma}
\newtheorem{proposition}{Proposition}
\newcommand{\la}{\langle}
\newcommand{\ra}{\rangle}

\newcommand{\Ext}{\operatorname{Ext}}
\renewcommand{\H}{\mathcal H}
\newcommand{\Hp}{\mathcal H_\Pi}
\newcommand{\Ht}{\H^{(2)}}
\newcommand{\Am}{A \operatorname{-mod}}
\newcommand{\Pm}{\Pi \operatorname{-mod}}
\newcommand{\Hom}{\operatorname{Hom}}
\newcommand{\Tor}{\operatorname{Tor}}
\newcommand{\Fac}{\operatorname{Fac}}
\DeclareFontFamily{U}{wncy}{}
    \DeclareFontShape{U}{wncy}{m}{n}{<->wncyr10}{}
    \DeclareSymbolFont{mcy}{U}{wncy}{m}{n}
    \DeclareMathSymbol{\Sh}{\mathord}{mcy}{"58} 
\newcommand{\Sp}{\operatorname{Split}}
\newcommand{\KzeroR}{K_0^*(A)_{\mathbb R}}
\newcommand{\KzeroPR}{K_0^*(\Pi)_{\mathbb R}}
\usetikzlibrary{shapes.geometric}
\title{Stability, shards, and preprojective algebras}
\author{Hugh Thomas}
\begin{document}
\maketitle
The goal of this note two-fold.  First, I would like to draw attention to the
way that stability gives us a geometrical picture of (some
of) the extension-closed abelian subcategories of a finite-dimensional
algebra.  Second, I will
describe Nathan Reading's \emph{shards} of
a hyperplane arrangement, and explain their relevance to understanding
the stability picture for 
finite-type preprojective algebras.  

\section{Semistable subcategories}

Let $A$ be a finite-dimensional algebra over a field $k$.  We will work
with the category of left $A$-modules.  Suppose that $A$ has $n$ pairwise
non-isomorphic simple modules $S_1,\dots,S_n$.

The Grothendieck group of $A$ can be defined as the free abelian group on
a set of generators corresponding to the simple modules:

$$K_0(A)=\bigoplus_{i=1}^n \mathbb Z[S_i].$$

For any $A$-module $M$, there is a corresponding class in $K_0(A)$, which
we denote $[M]$, and which is equal to $\sum_i c_i [S_i]$, where $c_i$ is the
number of times $S_i$ appears in a composition series for $M$.

We will be interested in linear functionals on $K_0(A)$.  For convenience
in drawing pictures, we will extend scalars to consider real-valued functionals.

$$\KzeroR= \Hom_{\mathbb Z} (K_0(A),\mathbb R)$$

Let $\phi \in \KzeroR$.  An $A$-module $M$ is called semistable with
respect to $\phi$ if $\phi([M])=0$ and $\phi([N])\leq 0$ for any 
submodule $N$ of $M$.
This definition is due to King \cite{Ki}, who showed that it is equivalent to a
notion of semistability coming from geometric invariant theory (which we
shall not need in this note).  There is also another reformulation in terms
of semi-invariants (see, for example, \cite{DW}),
but we shall not need to refer to this perspective either.  

We will write $(\Am)^\phi$ for the full subcategory of $A$-modules semistable
with respect to $\phi$.  It was shown by King (and it is an easy exercise)
that $(\Am)^\phi$ is an extension-closed, exact abelian subcategory of $\Am$.  


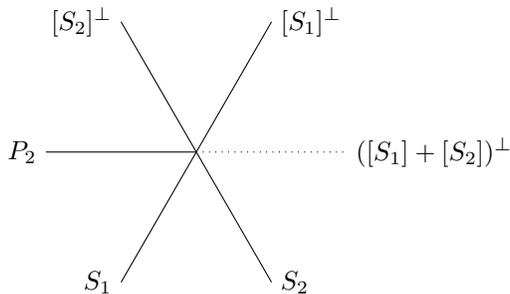
\begin{figure}[h]
$$\begin{tikzpicture}[scale=2]
  \draw [rotate=0,dotted] (0,0) -- (1,0) node[right]{$([S_1]+[S_2])^\perp$};
  \draw [rotate=60] (0,0)--(1,0) node[right]{$[S_1]^\perp$}; 
  \draw [rotate=120] (0,0)--(1,0) node[left]{$[S_2]^\perp$};
  \draw [rotate=180] (0,0)--(1,0) node[left]{$P_2$};
  \draw [rotate=240] (0,0)--(1,0) node[left]{$S_1$};
  \draw[rotate=300] (0,0)--(1,0) node[right]{$S_2$};
  \end{tikzpicture}$$
  \caption{\label{f-one}Regions of semistability for a path algebra of type $A_2$} 
\end{figure}

For example, we could consider a path algebra of
type $A_2$, with simples $S_1$ being projective and $S_2$ being injective,
as shown in Figure \ref{f-one}.  This is a picture of $\KzeroR\cong\mathbb R^2$.
Each point of the picture corresponds to a stability condition.    
There, $[S_1]^\perp$ designates the line consisting of elements of
$\KzeroR$ which vanish on $[S_1]$, and similarly for $[S_2]^\perp$ and
$([S_1]+[S_2])^\perp$.  We will consistently orient our stability diagrams so that the region where $\phi$ is positive on all the simples is at the bottom of the diagram.  
We have marked two lines and one half-line which
are regions of semistability for the indecomposable modules of this algebra.  $S_1$ is stable on the
whole line $[S_1]^\perp$, and similarly for $S_2$.  $P_2$, however, is
only stable on half of the line $([S_1]+[S_2])^\perp$, the half that is
drawn in solidly. On the other half, its submodule $S_1$ has $\phi([S_1])>0$,
which causes it to become unstable.  Note that the labels $S_1$, $S_2$, and $P_2$ in the diagram do not refer to specific points.  Rather, they label their corresponding regions of stability (lines, or, in the case of $P_2$, a half-line).
Generically, at a point not on any of the lines in the picture, the
subcategory of semistable submodules is the zero category.
Finally, at
the origin, every module is semistable.  
Pictures like this, in the hereditary case, have been studied by \cite{IOTW, Ch, IT, IPT}.  


Recall that an $A$-module is called a brick if its endomorphism ring is a
division algebra.  To understand the semistable subcategories
of $A$-mod, it is sufficient to understand semistability of bricks, by
the following lemma.  

\begin{lemma} An exact abelian extension-closed subcategory of $A$-mod is determined by the
  bricks it contains. \end{lemma}

\begin{proof}  Let $\mathcal C$ be an abelian extension-closed subcategory of
  $A$-mod, and let $\mathcal B$ be the set of bricks it contains.
  The statement of the lemma follows from the fact which we shall establish
  that the objects of
  $\mathcal C$ are exactly those $A$-modules which admit a filtration by
  modules in  $\mathcal B$.

  Clearly, any module filtered by modules in $\mathcal B$ is contained in
  $\mathcal C$, because $\mathcal C$ is extension-closed.  Conversely, let $X\in \mathcal C$.  If $X$ is a brick,
  it is in $\mathcal B$ and we are done.  Otherwise, $X$ admits a
  non-invertible endomorphism $\alpha$.  Now $X$ is isomorphic to
  the extension of
  the image of $\alpha$ by its kernel, both of which have smaller total dimension
  than $X$, and both of which are in $\mathcal C$ because it is an exact abelian subcategory, so we are done by induction. \end{proof}

Thus, if we want to understand the map from $\KzeroR$ to semistable subcategories, it suffices to understand, for each brick of $\Am$, the region of
$\KzeroR$ for which it is semistable.  The category $(\Am)^\phi$ will
consist of all modules filtered by the bricks that are semistable for
$\phi$.

Our goal in this paper is to describe this picture for finite-type
preprojective algebras.  This will require a detour into the theory of
hyperplane arrangements.  

\section{Preprojective algebras}

First, though, we introduce the
finite-type preprojective algebras.  Let $Q$ be a simply-laced
Dynkin quiver, with a set $Q_0=\{1,\dots,n\}$ of vertices and a set $Q_1$ of
arrows.  Define $\overline Q$ to be the doubled quiver of $Q$, which
is to say, for each arrow $a: i \rightarrow j$, we add an arrow $a^*: i \leftarrow j$.  The preprojective algebra is then defined to be:
$$\Pi=k\overline Q/\sum_{a\in Q_1} (aa^*-a^*a).$$
This is a finite-dimensional self-injective algebra.

Preprojective algebras
were originally introduced by Gelfand and Ponomarev \cite{GP}, and, in a version
closer to the formulation which is now standard, by Dlab and Ringel \cite{DR}.
They arise naturally in geometric representation theoretic contexts, playing,
for example, an essential role in Lusztig's definition of the semicanonical
basis of the enveloping algebra of the positive part of a symmetric Kac-Moody Lie algebra \cite{Lu}.  For our purposes, we can just take them as an
interesting class of algebras
with a Dynkin classification; as we shall see, other elements of Dynkin diagram
combinatorics will also turn out to be relevant to their analysis.

As a simple example, let us consider the preprojective algebra of type $A_2$.
We have two vertices $1$ and $2$, an arrow $a$ from 1 to 2, an arrow $a^*$ from 2 to 1, and the relation $aa^*-a^*a$.  Multiplying this relation on both sides
by the idempotent at 1, and multiplying on both sides by the idempotent at 2,
we deduce that the ideal generated by $aa^*-a^*a$
actually contains each of $aa^*$ and $a^*a$, so in this case, we could have
described the ideal of relations as being generated by $aa^*$ and $a^*a$.

Either by noticing that this implies that the preprojective algebra of type
$A_2$ happens to be a gentle algebra, or just by
thinking about it, we determine that this algebra has four indecomposable
modules: the simples at each vertex and the projectives at each vertex,
which are of length two.  All four of these modules are bricks.  In keeping
with the point of view developed in the previous section, we can ask ourselves
where these bricks are semistable.  The answer is given in Figure \ref{f-two}.

\begin{figure}[h]
  $$\begin{tikzpicture}[scale=2]
  \draw [rotate=0] (0.1,0) -- (1,0) node[right]{$P_1$};
  \draw [rotate=60] (0,0)--(1,0) node[right]{$[S_1]^\perp$}; 
  \draw [rotate=120] (0,0)--(1,0) node[left]{$[S_2]^\perp$};
  \draw [rotate=180] (0.1,0)--(1,0) node[left]{$P_2$};
  \draw [rotate=240] (0,0)--(1,0) node[left]{$S_1$};
  \draw[rotate=300] (0,0)--(1,0) node[right]{$S_2$};
\end{tikzpicture}$$
\caption{\label{f-two}Regions of semistability for the preprojective algebra of type $A_2$}
\end{figure}
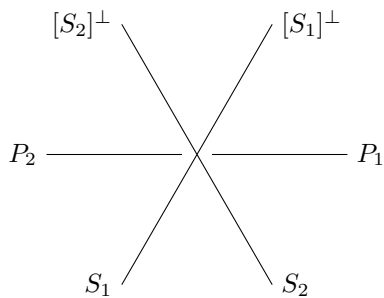

The verification that this picture is correct is essentially the same as
for the hereditary example examined above.
For clarity, the regions where $P_1$ and $P_2$ are semistable
are drawn as if they don't quite touch the origin,
but in fact they extend up
to and include it.

We notice that if we consider the union of the lines and half-lines where at least one brick is
semistable, this is a very symmetrical picture.  As we shall see, this is no
coincidence, but in order to make this notion precise, we shall have to
introduce some further technology: specifically, we shall have to introduce
the Weyl groups to provide the symmetries we want.  

\section{Weyl groups}
Good general references for Weyl groups are \cite{Hu,BB}.  

We want to define a bilinear form on $K_0(\Pi)$.  This can be defined very
explicitly by saying that $\la [S_i],[S_i]\ra = 2$, and for $j\ne i$,
$\la [S_i],[S_j] \ra$ is minus the number of arrows between vertices $i$ and
$j$ in $Q$.

A more conceptual definition is to consider the affine-type quiver $\hat Q$
which is obtained by adding a single vertex to $Q$.  (There is a unique way
to do this.)  We can then define the corresponding preprojective algebra
$\hat \Pi$, and the category of nilpotent $\hat \Pi$-modules, and consider its Grothendieck group and Euler form: for $V$ and
$W$ nilpotent $\hat \Pi$-representations,
\begin{eqnarray*}\la [V],[W]\ra &=& \sum_{i=0}^\infty (-1)^i\dim \Ext^i(V,W)\\
&=& \dim \Hom(V,W) - \dim \Ext^1(V,W) + \dim \Ext^2(V,W).\end{eqnarray*}
Restricted to the subspace spanned by the classes of the simple modules of $\Pi$,
we recover the form defined in the previous paragraph.  Note that
we cannot directly take the Euler form for $\Pi$ because its global dimension  is not finite, so $\sum_{i=0}^\infty(-1)^i \dim \Ext^i(V,W)$ is not
well-defined.

The bilinear form on $K_0(\Pi)$ turns out to be positive definite.  Since,
in particular, it is non-degenerate, we can use it to identify
$K_0(\Pi)\otimes \mathbb R$ with its dual, and
thus define a positive definite symmetric bilinear form
on $\KzeroPR$.  

We now want to define a group action on $V=\KzeroPR$.  
Let $e_1,\dots,e_n$ be
the standard basis for $V$, with $e_i([S_j])=\delta_{ij}$, where
$\delta_{ij}$ is 1 if $i=j$ and 0 otherwise.  Define a linear transformation
$s_i$ of $V$ by

$$s_i(\phi)= \phi- \la e_i,\phi \ra e_i$$

Each $s_i$ acts by reflecting in a hyperplane with respect to the bilinear
form on $V$.  We call the transformations $s_i$ 
\emph{simple reflections}.  

The group generated by these $n$ elements forms, by definition,
the Weyl group associated
to $Q$.  We denote it by $W$, and we think of it as acting on $V$ on the right.

For future use, for any $w\in W$,
we define $\ell(w)$ to be the length of the shortest possible
expression for $w$ as a product of the simple reflections.  The identity element
is the unique element of length zero, and the simple reflections are exactly
the elements of length one. 

Define $T$ to be the set of conjugates in $W$
of $s_1,\dots,s_n$.  It is easy to see that all of these elements will also act
by reflections.  In fact, they are all the reflections in $W$.  

By definition, the {\emph reflection arrangement} $\Hp$
associated to $W$ consists
of the collection of reflecting hyperplanes in $V$ associated to the set $T$ of
reflections.  

We can now state a rough version of our main result: the region in $\KzeroPR$ where $(\Pm)^\phi \ne 0$ consists of exactly the
  union of the hyperplanes in $\Hp$.  This accounts for the regularity
which we observed in the case of the $A_2$ preprojective algebra.
  In order to refine this result to get a picture like Figure \ref{f-two},
  which reflects where each
  brick is semistable, we will need some
  way to divide the reflecting hyperplanes up into pieces.  It turns out that
  a natural way to do this was developed, for superficially different
  purposes, by Reading \cite{Re1}, as we now explain.




  
\section{Shards}
We must now take a detour into the theory of hyperplane arrangements and,
in particular, the poset structure on the poset of regions defined by a
hyperplane arrangement.  The key results we need are to be found in
\cite{Re1}.  \cite{Re2} is an exposition which provides further context.  

Let $\H$ be a hyperplane arrangement in $\mathbb R^n$, by which we mean a collection of
finitely many linear hyperplanes in $\mathbb R^n$.  

$\H$ defines a set of \emph{chambers}, which are the connected components of
$\mathbb R^n \setminus \bigcup_{H\in \H} H$.  

There is a natural graph structure associated to $\H$, which we denote
$G(\H)$.  The vertices are the
chambers, and two chambers are adjacent if their closures intersect in a
codimension-one region in $\mathbb R^n$.  

We shall define a poset structure on the set of chambers by specifying
its \emph{cover relations},
that is to say, the pairs
$E,F$ such that $E<F$ and there is no $G$ with $E<G<F$.  We write
$E\lessdot F$ for a cover relation in a poset.  The Hasse diagram of a
poset is the directed graph on the elements of the poset, whose edges
$(E,F)$ are exactly the cover relations $E\lessdot F$ of the poset.  

Choose a base chamber and call it $C$.  
Define the chamber poset $P(\H,C)$ on the set of chambers by imposing that
$E\lessdot F$ if and only if $E$ and $F$ are adjacent and
$E$ lies on the same side as $C$ of the hyperplane defined by the intersection
of the closures of $E$ and $F$.

The Hasse diagram of $P(\H,C)$ is then an orientation of the graph $G(\H)$.
The chamber $C$ is the unique source, corresponding to the fact that it is
the minimum element of the poset, and the chamber $-C$ is the maximum element
of the poset.  (For any chamber $E$, note that $-E$ also forms a chamber.)
This poset was introduced by Edelman \cite{Ed}.

A chamber is called \emph{simplicial} if it consists of positive linear
combinations of $n$ linearly independent vectors.  $\H$ is called simplicial
if all its chambers are simplicial.  If $n\leq 2$, all hyperplane arrangements
are simplicial, but this is not true for larger $n$.

An important source of simplicial arrangements are the \emph{reflection
  arrangements}.  The reflection arrangement associated to $\Pi$, which we
have already introduced, is an example, but any finite Coxeter group yields a
reflection arrangement in the same way.
The reader who is unfamiliar with Coxeter groups
may simply take our statements about reflection arrangements as applying to
the reflection arrangements we have already introduced, with no loss.  


Let $\H$ be a reflection arrangement.  
$W$ acts
on the set of chambers simply-transitively, so, after identifying the base
chamber $C$ with the identity element of $W$, we can identify the chambers
with the elements of $W$.  The poset $P(\H,C)$ is then a well-known poset
on $W$, known as \emph{(right) weak order}, in which the cover relations are
given by $v\lessdot vs_i$ if $\ell(vs_i)=\ell(v)+1$.  

A poset is called a \emph{lattice} if any pair of elements $E,F$ has a
unique greatest lower bound, denoted $E\wedge F$ (the \emph{meet} of $E$ and $F$), and a unique least upper
bound, denoted $E\vee F$ (the \emph{join} of $E$ and $F$).  

\begin{theorem}[{\cite[Theorem 3.4]{BEZ}}] If $\H$ is a simplicial arrangement, then
  the poset $P(\H,C)$ is a lattice. \end{theorem}

Simplicialness is not necessary for $P(\H,C)$ to be a lattice,
see \cite{Re2}.  However,
since our eventual application will be to reflection arrangements, we may
as well not seek the greatest possible generality.

A lattice $L$ is called semi-distributive if for $E,F,G$ in $L$ such that $E\vee F=E\vee G$ it follows that this element also equals $E\vee(F\wedge G)$, and
dually if $E\wedge F=E\wedge G$, then this element also equals $E\wedge(F\vee G)$.  We have the following result:

\begin{theorem}[{\cite[Corollary 9-3.9]{Re2}}] If $\H$ is a simplicial arrangement, then the lattice $P(\mathcal H,C)$ is semidistributive. \end{theorem}

It is an immediate consequence of
semidistributivity that if $G>E$ then there is a unique minimum element
among all elements $F$ such that $E\vee F=G$.  

An element $E$ of a lattice $L$ is called \emph{join-irreducible} if it
is not the minimum element of the lattice, and it cannot
be written as $E=F\vee G$ with $F,G<E$.  The following lemma is an easy
exercise.  

\begin{lemma} \label{unique} If $L$ is a finite lattice, then
  $E$ is join-irreducible in $L$ if and only if $E$ covers exactly one element.  \end{lemma}

If $E$ is join-irreducible, we write $E_*$ for the unique element which it covers.

In a finite lattice, every element can be written as a join of join-irreducible
elements, so they have an obviously important structural rôle.  (This rôle
is shared with the meet-irreducible elements, which are defined dually, and
can be studied in the same way as we are doing for join-irreducible
elements.)

There is a natural labelling of the edges of the Hasse diagram of $P(\H,C)$
by join-irreducible elements, as follows: define the join-irreducible
label $j(E\gtrdot F)$ to be the minimum $G$ such that $E\vee G=F$.  By semidistributivity,
this is well-defined, and it is clear that it must be join-irreducible.

Given the importance of join-irreducible elements of a lattice, 
it is natural to ask how to see the join-irreducible elements
of $P(\H,C)$ in terms of the geometry of $\H$.  

Each join-irreducible of $P(\H,C)$ is naturally associated to a particular
hyperplane.  Namely, if $E$ is join-irreducible, then by Lemma \ref{unique},
it covers a unique other chamber $F$, and by the definition of the cover
relations in $P(\H,C)$, the span of
intersection of the closures of $E$ and $F$ defines a hyperplane of $\H$.

This map from join-irreducible elements of $P(\H,C)$ to $\H$ is not a bijection,
as demonstrated by Figure \ref{f-three}, which show the Hasse diagram of the poset of regions of a two-dimensional hyperplane arrangement superimposed over the hyperplane arrangement.  We always draw the base chamber at the bottom.  
The join-irreducible elements are marked with black
dots, and the arrows indicate the map from join-irreducibles to hyperplanes.
We see that there are two join-irreducible elements which are associated
to the hyperplanes $H_2$ and $H_3$.

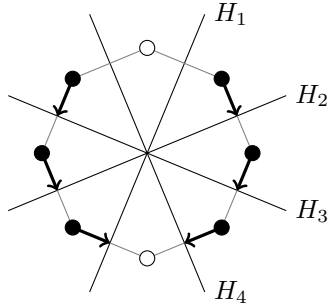
\begin{figure}[h]
  $$\begin{tikzpicture}[scale=2]
  \draw [rotate=45/2] (0,0) -- (1,0) node[right]{$H_2$};
  \draw [rotate=135/2] (0,0)--(1,0) node[right]{$H_1$}; 
  \draw [rotate=90+22.5] (0,0)--(1,0);
  \draw [rotate=180-22.5] (0,0)--(1,0);
  \draw [rotate=180+22.5] (0,0)--(1,0);
  \draw[rotate=270-22.5] (0,0)--(1,0);
    \draw[rotate=270+22.5] (0,0)--(1,0) node[right]{$H_4$};
    \draw[rotate=-22.5] (0,0)--(1,0) node[right]{$H_3$};
    \draw[gray] (90:.7) -- (135:.7) -- (180:.7) -- (225:.7) -- (270:.7) -- (315:.7) --
    (0:.7) -- (45:.7) -- (90:.7);
    \draw[very thick,->](45:.7)--(22.5:.7*.9239);
        \draw[very thick,->](-45:.7)--(-45-22.5:.7*.9239);
    \draw[very thick,->](0:.7)--(-22.5:.7*.9239);
    \draw[very thick,->](180-45:.7)--(180-22.5:.7*.9239);
    \draw[very thick,->](180+45:.7)--(180+45+22.5:.7*.9239);
    \draw[very thick,->](180:.7)--(180+22.5:.7*.9239);

    \draw[fill=white] (90:.7) circle (0.05);
   \draw[fill=black] (135:.7) circle (0.05);
   \draw[fill=black] (180:.7) circle (0.05);
   \draw[fill=black] (225:.7) circle (0.05);
   \draw[fill=white] (270:.7) circle (0.05);
   \draw[fill=black] (315:.7) circle (0.05);
   \draw[fill=black] (0:.7) circle (0.05);
   \draw[fill=black] (45:.7) circle (0.05);
  \end{tikzpicture}$$
\caption{\label{f-three}Poset of regions of a two-dimensional hyperplane arrangement}
  \end{figure}

To define a bijection from join-irreducible elements to something geometric,
Reading was impelled to split some of the hyperplanes in two, as in Figure
\ref{f-four}. Now, each join-irreducible element (black dot) has a distinct hyperplane
or half-hyperplane directly below it.  

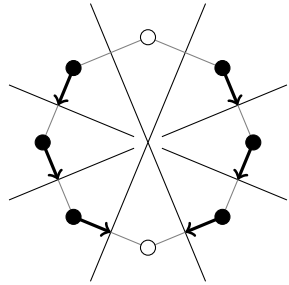
\begin{figure}[h]
  $$\begin{tikzpicture}[scale=2]
  \draw [rotate=45/2] (0.1,0) -- (1,0);
  \draw [rotate=135/2] (0,0)--(1,0); 
  \draw [rotate=90+22.5] (0,0)--(1,0);
  \draw [rotate=180-22.5] (0.1,0)--(1,0);
  \draw [rotate=180+22.5] (0.1,0)--(1,0);
  \draw[rotate=270-22.5] (0,0)--(1,0);
    \draw[rotate=270+22.5] (0,0)--(1,0);
    \draw[rotate=-22.5] (0.1,0)--(1,0);
    \draw[gray] (90:.7) -- (135:.7) -- (180:.7) -- (225:.7) -- (270:.7) -- (315:.7) --
    (0:.7) -- (45:.7) -- (90:.7);
    \draw[very thick,->](45:.7)--(22.5:.7*.9239);
        \draw[very thick,->](-45:.7)--(-45-22.5:.7*.9239);
    \draw[very thick,->](0:.7)--(-22.5:.7*.9239);
    \draw[very thick,->](180-45:.7)--(180-22.5:.7*.9239);
    \draw[very thick,->](180+45:.7)--(180+45+22.5:.7*.9239);
    \draw[very thick,->](180:.7)--(180+22.5:.7*.9239);

   \draw[fill=white] (90:.7) circle (0.05);
   \draw[fill=black] (135:.7) circle (0.05);
   \draw[fill=black] (180:.7) circle (0.05);
   \draw[fill=black] (225:.7) circle (0.05);
   \draw[fill=white] (270:.7) circle (0.05);
   \draw[fill=black] (315:.7) circle (0.05);
   \draw[fill=black] (0:.7) circle (0.05);
   \draw[fill=black] (45:.7) circle (0.05);
  \end{tikzpicture}
  $$
  \caption{\label{f-four}Splitting hyperplanes when $n=2$}
\end{figure}

More formally, when $n=2$,
Reading splits in two the hyperplanes which are not adjacent
to the base chamber, and calls this set of hyperplanes and half-hyperplanes
the \emph{shards} of $\H$.  
Now, if $E$ is a join-irreducible element of $P(\H,e)$,
we see that the facet of $E$ corresponding to the unique cover $E\gtrdot F$
lies in a well-defined shard of $\H$, and this gives us a bijection from
join-irreducible elements to shards.  

The fact that this works for  $n=2$ is a rather trivial observation.
The surprising fact is that
this simple strategy of splitting up hyperplanes is exactly what is needed
in general.  

To define the general strategy, we need to introduce some further notation.  Let
$\Ht$ be the set of codimension-two intersections of hyperplanes from
$\H$, i.e., $$\Ht=\{H\cap K\mid H,K \in \H, H\ne K\}.$$

For each $X\in \Ht$, consider the hyperplanes in $\H$ containing $X$.
Note that since $C$ is a chamber of our original arrangement, and we are
now considering the sub-arrangement of just those hyperplanes that contain
$X$, the chamber $C$ is located in a particular chamber of this sub-arrangement.
Number the hyperplanes containing $X$ cyclically as $H_1, H_2, \dots, H_r$
so that $C$ is between $H_1$ and $H_r$, as in Figure \ref{f-five}.

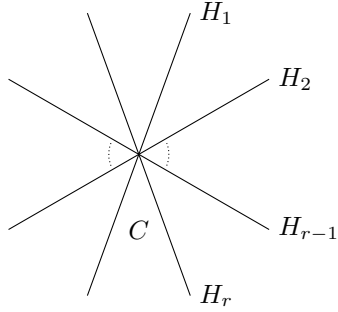
\begin{figure}[h]
  $$\begin{tikzpicture}[scale=2]
  \draw [rotate=90-60] (0,0) -- (1,0) node[right]{$H_2$};
  \draw [rotate=90-20] (0,0)--(1,0) node[right]{$H_1$}; 
  \draw [rotate=90+20] (0,0)--(1,0);
  \draw [rotate=90+60] (0,0)--(1,0);
  \draw [rotate=270-60] (0,0)--(1,0);
  \draw[rotate=270-20] (0,0)--(1,0);
    \draw[rotate=270+20] (0,0)--(1,0) node[right]{$H_r$};
    \draw[rotate=270+60] (0,0)--(1,0) node[right]{$H_{r-1}$};
    \draw[densely dotted] (-22:.2) arc (-22:22:.2);
    \draw[densely dotted] (180-22:.2) arc (180-22:180+25:.2);
    \draw (0,-.5) node{$C$};
  \end{tikzpicture}$$
  \caption{\label{f-five}Numbering hyperplanes around a codimension 2 intersection}
\end{figure}
  
The idea is that, around $X$, we will split the hyperplanes as in the $n=2$
situation previously discussed.  We want $X$ to split the hyperplanes
which are not adjacent to $C$, so we define $\Sp(X)=\{H_2,\dots,H_{r-1}\}$.  

Now, a hyperplane $H\in \H$ is split into a set of shards, which we denote $\Sh_H$, by defining
$$\Sh_H= \textrm{ the components of } \left(H\setminus \bigcup_{X \mid H\in \Sp(X)} X \right)$$
and
$$\Sh(\H)=\bigcup_{H\in \H} \Sh_H.$$

Given a cover relation $E\gtrdot F$ in $P(\H,C)$, the
intersection of the two chambers $E$ and $F$ (which is a cone in the
hyperplane separating them) lies entirely in one shard.  We can therefore
define $\Sh(E\gtrdot F)$ to be this shard.   

Reading proved:

\begin{theorem}{\cite[Proposition 3.3]{Re1}}
  The map from join-irreducible elements of $P(\H,e)$ to shards, sending
  a join-irreducible $G$ to $\Sh(G\gtrdot G_*)$, is a bijection.  
\end{theorem}

Further, we have the following theorem (closely related to statements in
\cite{Re1,Re2}, but expressed in a way that is convenient for us):

\begin{theorem}\label{rephrase}
  The map sending $G$ to $\Sh(G\gtrdot G_*)$, sends the label $j(E\gtrdot F)$ to
  the label $\Sh(E\gtrdot F)$ for any $E\gtrdot F$.
\end{theorem}

\begin{proof}
  Let $G$ be the join-irreducible corresponding to the shard separating
  $E$ and $F$.  Since $E$ is above that shard, we have $G\leq E$ by
  \cite[Lemma 3.5]{Re1}.  Thus
  $G\vee F=E$.  Any element below $G$ is below the hyperplane separating
  $E$ and $F$.  Thus $G$ is a minimal element among those which join with
  $F$ to give $E$.  Since $P(\H,e)$ is semidistributive, $G$ must be the
  minimum element, and thus $G=j(E\gtrdot F)$, and we have that
  the shard associated to $j(E\gtrdot F)$ is indeed $\Sh(E\gtrdot F)$.  
\end{proof}

\section{Join-irreducibles of $W$ and bricks of $\Pi$}

In \cite{IRRT}, we constructed a bijection between join-irreducible elements
of $W$ and bricks of $\Pi$.  The simplest way to state it is as follows.
Let $e_i$ be the idempotent of $\Pi$ corresponding to the vertex $i$.
Define the two-sided ideal $I_i=\Pi(1-e_i)\Pi$.

Consider a word $\underline{w}=(i_1,\dots,i_r)$ with each
$i_j\in \{1,\dots,n\}$.  Define $I_{\underline{w}}= I_{i_1}\dots I_{i_r}$.
We say that $(i_1,\dots,i_r)$ is a reduced word for $w\in W$ if
$w=s_{i_1}\dots s_{i_r}$ and this is an expression for $w$ of the minimum
possible length.  

\begin{proposition}[\cite{IR}] If $\underline{w}_1$ and $\underline{w}_2$
  are reduced words for $w$, then $I_{\underline{w}_1}=I_{\underline{w}_2}$.
  \end{proposition}

  We can therefore define $I_w$ to be the ideal $I_{\underline w}$ where
  $\underline w$ is any reduced word for $w$.

  Let $w \gtrdot u$ be a cover relation in weak order.
  Following \cite{IRRT}, we define the brick label for $B(w\gtrdot u)$
  to be $I_u/I_w$.  This module turns out to be, indeed, a brick.  

  In \cite{IRRT}, we also consider the join-irreducible labelling
  $j(w\gtrdot u)$.  The definition used there is not the same as the one
  given here, but they are equivalent by \cite[Proposition 2.1]{IRRT}.



%

One of the main results of \cite{IRRT} can be stated as follows:
  
  \begin{theorem}[{\cite[Theorem 1.3]{IRRT}}] The map from join-irreducibles of $W$ to bricks
    of $\Pi$ sending $w$ to $B(w\gtrdot w_*)$ is a bijection which
    transforms the join-irreducible labelling into the brick labelling.
  \end{theorem}

  We can now state the main theorem of this note:

  \begin{theorem}\label{main}
    For $w$ a join-irreducible of $W$, the region where the
    brick $B(w\gtrdot w_*)$ is semistable is the closure of the shard
    $\Sh(w\gtrdot w_*)$.
  \end{theorem}

\section{Technical Lemmas}
  
  Before we begin the proof of the main theorem, we need a few technical lemmas.

  \begin{lemma}\label{reflects} Let $M$ be a $\Pi$-module such that
    $\Hom(S_i,M)=0$.  Then $[I_i\otimes M]=s_i([M])$.
    \end{lemma}
\begin{proof}
  If $\Hom(S_i,M)=0$, then $I_i\otimes M$ is isomorphic to the result of applying
  a certain spherical twist functor to $M$, where $M$ is thought of in the derived category of the corresponding affine-type preprojective algebra $\hat \Pi$.  (This part of
  the conclusion of \cite[Proposition 3.2(b)]{IRRT} follows if we assume
  only that $\Hom(S_i,M)=0$, although there, an additional homological
  assumption on $M$ is made.)
  Spherical twists act like reflections on the level of the
  Grothendieck group.  (See for example \cite[Lemma 2.6]{AIRT}.)
  \end{proof}
  
  \begin{lemma}\label{moving-up} Let $w\gtrdot u$ be a cover in weak order on $W$.  Let
    $i$ be such that $\ell(s_iw)>\ell(w)$.  Then $\Hom(S_i,B(w\gtrdot u))=0$.
\end{lemma}

  \begin{proof} The Weyl group element $u$ determines a torsion class
    $T_u=\Fac I_u$, and a corresponding torsion-free class $F_u$.  Because
    $\ell(s_iu)>\ell(u)$, $I_{s_iu}$ is properly contained in $I_u$, and thus
    $S_i$ is in the top of $I_u$, and in particular, $S_i\in T_u$.  On the
    other hand, by \cite[Theorem 4.5]{IRRT}, $B(w\gtrdot u) \in F_u$.
    Thus $\Hom(S_i,B(w\gtrdot u))=0$.  
    \end{proof}
  
  We remark that under the hypotheses of Lemma \ref{moving-up},
  since $B(s_iw \gtrdot s_iu)\cong I_i\otimes B(w\gtrdot u)$,
  what Lemma \ref{moving-up} says is that Lemma \ref{reflects} applies,
  so that
    $[B(s_iw \gtrdot s_iu)]=s_i([B(w\gtrdot u)])$.  This is
    part of what is
    implied by Theorem \ref{main}; see also \cite[Theorem 2.7(1)]{AIRT}.  

  \begin{lemma}\label{defect} Let $N$ be a submodule of $M$, and suppose that
    $\Hom(S_i,M)=0$.  Then the kernel of the induced map from $I_i\otimes N$ to
    $I_i\otimes M$ is a sum of some number (possibly
    zero) of copies of $S_i$.
  \end{lemma}

  \begin{proof} From the short exact sequence
    $$ 0 \rightarrow N \rightarrow M \rightarrow M/N \rightarrow 0$$
    we obtain
    $$ \Tor_1(I_i,M/N) \rightarrow I_i\otimes N \rightarrow I_i \otimes M
    \rightarrow I_i\otimes M/N \rightarrow 0$$
    To evaluate $\Tor_1(I_1,M/N)$, we can take
    $$0 \rightarrow I_i \rightarrow \Pi \rightarrow S_i \rightarrow 0$$
    and tensor by $M/N$, obtaining that $\Tor_1(I_i,M/N) \cong \Tor_2(S_i,M/N)$.
    As a $\Pi$-module, $\Tor_2(S_i,M/N)$ is congruent to a sum of some number
    of copies of $S_i$.
  \end{proof}

\section{Proof of Main Theorem}
  
  \begin{proof}[Proof of Theorem \ref{main}]
    Let $w=us_i \gtrdot u$ be a cover in weak order on $W$.  We will prove by
    reverse induction on the length of $w$ that $B(w\gtrdot u)$ is semistable
    on the (closed) facet of the Coxeter fan corresponding to $w\gtrdot u$.

    There is a unique element of $W$ of maximal length, usually denoted
    $w_0$, and the chamber
    corresponding to it is $-C$.  The hyperplanes that bound it are
    perpendicular to the simple roots, and the modules corresponding to the
    covers are the simple modules, each of which is semistable on its
    entire perpendicular hyperplane.  This establishes the base case of the
    induction.

    Now suppose that $w<w_0$.  Let $s_j$ be a simple reflection such that
    $\ell( s_jw)>\ell(w)$.  Let $w'=s_jw$, $u'=s_ju$, $B'=B(w'\gtrdot u')$.
    $B'$ is related to $B$ by $B'=I_j \otimes B$.  By Lemma \ref{moving-up}
    and Lemma \ref{reflects}, we have that
    $[B']=s_i([B])$.  
    
    Suppose that $B$ is not semistable for some $\phi$ in the facet
    corresonding to $w\gtrdot u$.
        This must be because of some subobject
        $E$ of $B$ such that $\phi([E])>0$.
Define $\phi'=s_i(\phi)$.  It falls on the
    facet $s_jw\gtrdot s_ju$.  
        We want to conclude that there is
        a corresponding destabilizing subobject of $B'$ for $\phi'$,
        which would contradict
    our induction hypothesis.  

    By Lemma \ref{moving-up}, $\Hom(S_i,B)=0$.  It therefore follows that
    $\Hom(S_i,E)=0$, so we can apply Lemma \ref{reflects} to conclude that
    $[I_i\otimes E]=s_i([E])$.  Therefore $\phi'([I_i\otimes E])=\phi([E])>0$.
    Let $E'$ be the image of $I_i\otimes E$ in
    $B'$.  
    The kernel of the natural map from $I_i \otimes E$ to $E'$
    is a sum of copies of $S_i$ by
    Lemma \ref{defect}.  Since $\ell(s_iu)>\ell(u)$, the chamber of $u$ lies
    on the opposite side from $C$ of the hyperplane perpendicular to $[S_i]$.
    Thus, $\phi'([S_i])\leq0$, so $\phi'([E'])\geq\phi'([I_i\otimes E])>0$.
    It follows that $E'$ is destabilizing for $B'$ with respect to
    $\phi'$, which is contrary to our induction hypothesis.  Therefore
    $B(w\gtrdot u)$ is semistable with respect to weights on the facet
    corresponding to $w\gtrdot u$, as desired.  

    Now we prove the opposite direction, namely, that a brick must be
    unstable outside the closure of the corresponding shard.  
    Let $w$ be a join-irreducible of $W$.  Let
    $\Sh=\Sh(w\gtrdot w_*)$, $B=B(w\gtrdot w_*)$.  Consider a facet $X$ of
    $\Sh$.  By the construction of shards, the span of $X$ is a
    codimension two intersection in $\Ht$, and around $X$ we have a picture
    with four shards, as shown in Figure \ref{f-six}. As
    always, the base chamber $C$ is at the bottom.  

    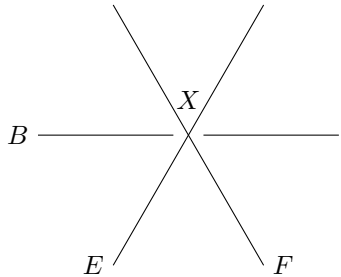
\begin{figure}[h]
$$\begin{tikzpicture}[scale=2]
  \draw [rotate=0] (0.1,0) -- (1,0);
  \draw [rotate=60] (0,0)--(1,0); 
  \draw [rotate=120] (0,0)--(1,0);
  \draw [rotate=180] (0.1,0)--(1,0) node[left]{$B$};
  \draw [rotate=240] (0,0)--(1,0) node[left]{$E$};
  \draw[rotate=300] (0,0)--(1,0) node[right]{$F$};
  \draw (0,0.11) node[above]{$X$};
      \end{tikzpicture}$$
      \caption{\label{f-six}The local picture around $X$}
\end{figure}

    By \cite[Proposition 4.3]{IRRT}, if $E$ and $F$ are the bricks
    associated to the shards
    as in the picture, then there is a short exact sequence:
    $$0 \rightarrow E \rightarrow B \rightarrow F \rightarrow 0. $$

    For $\phi \in C$, we have that $\phi([S_i]) >0$ for all
    $i$.  Let $\theta \in \Sh$.  Since $\Sh$ is on the opposite side of
    $[E]^\perp$ from $E$, $\theta([E])\leq 0$.  This is consistent with
    the fact which we have already established that $B$ is $\theta$-semistable.
    However, if $\theta$ is strictly on the opposite side of the hyperplane in
    $[B]^\perp$ defined by the span of $X$, then $\theta([E])> 0$,
    so $B$ is not $\theta$-semistable.  This establishes the theorem.
  \end{proof}

  


  \section{Connection to other work}
  Baumann, Kamnitzer, and Tingley \cite{BKT} study the representation theory
  of preprojective algebras of affine type.  Many of the ideas from this note
  could also be extracted from their work, but they do not discuss shards, so
  the combinatorics we present here is less explicitly developed.  
  
  Crawley-Boevey establishes a result about the existence of representations
  of deformed preprojective algebras \cite[Theorem 1.2]{CB-geom} which
  implies 
  that $(\Pm)^\phi \ne 0$ iff $\phi$ lies on a
  reflecting hyperplane by \cite[Lemma 3]{CB-exc}.  However, the argument to
  pass from the deformed preprojective algebra to semistable representations
  of the usual preprojective algebras depends on an assumption that the ground
  field is the complex numbers.  
  
  \section*{Acknowledgements}
  I would like to thank David Speyer, who gave me the suggestion that there should be a connection between shards and stability.  Discussions with him were crucial to the development of my understanding of this topic.  The work presented here is part of an ongoing larger joint project.  I would also like to thank
  Osamu Iyama, Nathan Reading, and Idun Reiten; this paper also draws greatly
  on things I learned from our collaboration.  I am grateful to the organizing
  committee of ICRA 2016 for having invited me to give the talk the details
  of which are fleshed out here, and to the referee for his or her comments.
  I would also like to record my gratitude to
  the Université Paris VII for the excellent working conditions in which I
  finished writing up this paper.  I gratefully acknowledge financial support from
  NSERC and the Canada Research Chairs program.


\begin{thebibliography}{IOTW}
  \bibitem[AIRT]{AIRT} C. Amiot, O. Iyama, I. Reiten, and G. Todorov,
Preprojective algebras and $c$-sortable words. \emph{Proc. Lond. Math. Soc. (3)} \textbf{104} (2012), 513--539.
  \bibitem[BKT]{BKT} P. Baumann, J. Kamnitzer, and P. Tingley.
    Affine Mirković-Vilonen polytopes. \emph{Publ. Math. Inst. Hautes Études Sci.} \textbf{120} (2014), 113--205.
  \bibitem[BB]{BB} A. Björner and F. Brenti, \emph{Combinatorics of Coxeter
    groups}, Springer, New York, 2005.
  \bibitem[BEZ]{BEZ} A. Björner, P. Edelman, and G. Ziegler. Hyperplane
    arrangements with a lattice of regions.  \emph{Discrete and Computational
      Geometry} \textbf{5} (1990), 263--288.
  \bibitem[Ch]{Ch} C. Chindris. Cluster fans, stability conditions, and domains of semi-invariants. \emph{Trans. Amer. Math. Soc} \textbf{363} (2011), 2171--2190.
  \bibitem[CB]{CB-exc} W. Crawley-Boevey.  On the exceptional fibres of
    Kleinian singularities. \emph{Amer. J. Math.} \textbf{122} (2000),
    1027--1037.
  \bibitem[CB2]{CB-geom} W. Crawley-Boevey.  Geometry of the moment map for
    representations of quivers.  \emph{Compositio Mathematica} \textbf{126}
    (2001), 257--293.  
\bibitem[DW]{DW} H. Derksen and J. Weyman. The combinatorics of quiver representations. \emph{Ann. Inst. Fourier} \textbf{61} (2011), 1061--1131.
    \bibitem[DR]{DR} V. Dlab and C. M. Ringel. The preprojective algebra of a modulated graph. in \emph{Representation theory, II (Proc. Second Internat. Conf., Carleton Univ., Ottawa, Ont., 1979)}, Springer, Berlin, 1980, pp. 216--231, 
    \bibitem[Ed]{Ed} P. Edelman. A partial order on the regions of $\mathbb R^n$
      dissected by hyperplanes.  \emph{Trans. Amer. Math. Soc.} \textbf{283}
      (1984), 617--631.
    \bibitem[GP]{GP} I. Gel'fand and V. Ponomarev.  Model algebras and
      representations of graphs.  \emph{Func. Anal. Appl.} \textbf{13} (1979)
      157--166.
    \bibitem[Hu]{Hu} J. Humphreys, \emph{Reflection
      groups and Coxeter groups}, Cambridge University Press, Cambridge, 1990. 
    \bibitem[IOTW]{IOTW} K. Igusa, K. Orr, J. Weyman, and G. Todorov. Cluster complexes via semi-invariants.\emph{Compositio Mathematica} \textbf{145} (2009),
      1001--1034.
    \bibitem[IT]{IT} C. Ingalls and H. Thomas. Noncrossing partitions and representations of quivers. \emph{Compositio Mathematica} \textbf{145} (2009), 1533--1562.
    \bibitem[IPT]{IPT} C. Ingalls, C. Paquette, and H. Thomas.
      Semistable subcategories for Euclidean quivers.  \emph{Proc. Lond.
        Math. Soc. (3)} \textbf{110} (2015), 805--840.
      \bibitem[IRRT]{IRRT} O. Iyama, N. Reading, I. Reiten, and H. Thomas.
        Lattice structure of Weyl groups via representation theory of
        preprojective algebras.  arXiv:1604.08401.  
      \bibitem[IR]{IR} O. Iyama and I. Reiten. Fomin-Zelevinsky mutation and tilting modules over Calabi-Yau algebras.  \emph{Amer. J. Math.} \textbf{130}
        (2008), 1087--1149.  
        \bibitem[Lu]{Lu} G. Lusztig. Semicanonical bases arising from enveloping algebras.  \emph{Advances in Mathematics} \textbf{151} (2000), 129--139.
        \bibitem[Ki]{Ki} A. King.    Moduli of representations of finite-dimensional algebras.  \emph{Quart. J. Math. Oxford Ser. (2)} \textbf{45} (1994),
          no. 180, 515--530.
  \bibitem[Re1]{Re1} N. Reading.  Noncrossing partitions and the shard intersection
    order. \emph{J. Algebraic Combin.} \textbf{33} (2011), 483--530.
  \bibitem[Re2]{Re2} N. Reading.  Lattice theory of the poset of regions.
    in \emph{Lattice Theory: Special Topics and Applications}, G. Grätzer and
    F. Wehrung, eds., Birkhäuser, Basel, 2016,  pp. 399--487.
  \end{thebibliography}
\end{document}